\documentclass[pagebackref,colorlinks,citecolor=blue,linkcolor=blue,urlcolor=blue,filecolor=blue]{article}
\pdfoutput=1
\usepackage[margin=1in]{geometry}
\usepackage{tikz-cd}
\usetikzlibrary{calc}
\usepackage{float}
\usepackage{amsmath}
\usepackage{amsfonts}
\usepackage{amsthm}
\usepackage{enumitem}
\usepackage{dsfont}
\usepackage{amssymb}
\usepackage{pifont}
\usepackage{mathtools}
\usepackage{comment}
\usepackage{graphicx, caption} 
\usepackage{float}
\usetikzlibrary{matrix}
\usepackage[T1]{fontenc}
\usepackage[utf8]{inputenc}

\usepackage{pinlabel} 

\usepackage{hyperref}

\newtheorem*{T1}{Theorem~\ref{seq top complexity of conf(n,w)}}
\newtheorem*{T2}{Theorem~\ref{seq dist top complexity of conf(n,w)}}
\newtheorem*{T3}{Theorem~\ref{seq dist top complexity of FnR2}}

\newtheorem{thm}{Theorem}[section]
\newtheorem{lem}[thm]{Lemma}
\newtheorem{prop}[thm]{Proposition}

\newtheorem{cor}[thm]{Corollary}
\newtheorem{exam}[thm]{Example}
\theoremstyle{definition}

\theoremstyle{definition}

\newtheorem*{claim*}{Claim}
\newtheorem*{quest*}{Question}
\newtheorem*{remark*}{Remark}
\newtheorem*{fact*}{Fact}

\newcommand{\thmtext}{\[
\textbf{TC}_{\textbf{r}}\big(\text{conf}(n,w)\big)=\begin{cases}0&\mbox{if }n=1,\\
r(n-1)-1&\mbox{if }1<n\le w,\\
r\big(n-\big\lceil\frac{n}{w}\big\rceil\big)&\mbox{if }n>w.
\end{cases}
\]}

\newcommand{\thmtextone}{\[
\textbf{dTC}_{\textbf{r}}\big(\text{conf}(n,w)\big)=\begin{cases}0&\mbox{if }n=1,\\
r(n-1)-1&\mbox{if }1<n\le w,\\
r\big(n-\big\lceil\frac{n}{w}\big\rceil\big)&\mbox{if }n>w.
\end{cases}
\]}

\newcommand{\thmtexttwo}{\[
\textbf{dTC}_{\textbf{r}}\big(F_{n}(\R^{2})\big)=r(n-1)-1.
\]}

\newcommand{\Z}{\ensuremath{\mathbb{Z}}}
\newcommand{\Q}{\ensuremath{\mathbb{Q}}}
\newcommand{\R}{\ensuremath{\mathbb{R}}}

\newcommand{\F}{\ensuremath{\mathbb{F}}}

\title{The sequential (distributional) topological complexity of the ordered configuration space of disks in a strip}
\author{Nicholas Wawrykow}
\date{}

\begin{document}
\maketitle

\begin{abstract}
How hard is it to program $n$ robots to move about a long narrow aisle while making a series of $r-2$ intermediate stops, provided only $w$ of the robots can fit across the width of the aisle?
In this paper, we answer this question by calculating the $r^{\text{th}}$-sequential topological complexity of $\text{conf}(n,w)$, the ordered configuration space of $n$ open unit-diameter disks in the infinite strip of width $w$, as well as its $r^{\text{th}}$-sequential distributional topological complexity.
We prove that as long as $n$ is greater than $w$, the $r^{\text{th}}$-sequential (distributional) topological complexity of $\text{conf}(n,w)$ is $r\big(n-\big\lceil\frac{n}{w}\big\rceil\big)$.
This shows that any non-looping program moving the $n$ robots between arbitrary initial and final configurations, with $r-2$ intermediate stops, must consider at least $r\big(n-\big\lceil\frac{n}{w}\big\rceil\big)$ cases.
\end{abstract}

\section{Introduction}
When topological complexity was first introduced in the early 2000s, many of the motivating examples came from robotics, e.g., the configuration spaces of robot arms \cite[Section 8]{farber2003topological} and of swarms of robots moving along tracks \cite{farber2005collision}.
To arrive at these tractable examples, certain idealizations were made: The robots were assumed to take up no space, allowing them to be treated as points, and they were only able to move along certain predefined paths.
At the time---and this is still true to an extent today---real-world computational and economic constraints made the latter assumption perfectly natural; it is much easier and cheaper to program robots to move along tracks than it is to give them two or more dimensions of freedom. 
The former assumption was more of a mathematical one.
Point configuration spaces had been a popular object of study for several decades at that point, and disk configurations had yet to be introduced; moreover, if the underlying space is large enough, then the corresponding point and disk configuration spaces are homotopy equivalent.
Still, there are many situations where these idealizations do not hold.
Fortunately, recent computational advances and the development of disk configuration spaces allow us to do away with these assumptions in certain situations.

To return to topological complexity's roots, we study a space that naturally arises if one considers $n$ circular robots moving about a long aisle such that only $w$ of the robots can fit abreast, i.e., we consider $\text{conf}(n,w)$ the \emph{ordered configuration of $n$ open unit-diameter disks in the infinite strip of width $w$}.
This disk configuration space, which can be thought of a subspace of $\R^{2n}$,
\[
\text{conf}(n, w):=\big\{(x_{1}, y_{1}, \dots, x_{n}, y_{n})\in \R^{2n}|(x_{i}-x_{j})^{2}+(y_{i}-y_{j})^{2}\ge 1\text{ for }i\neq j\text{ and }\frac{1}{2}\le y_{i}\le w-\frac{1}{2} \text{ for all } i\big\},
\]
is a well-studied variation of the ordered configuration space of $n$ points in the plane, e.g., \cite{alpert2021configuration, alpert2021configuration1, BBK, wawrykow2022On, wawrykow2023representation}.
We compute two variations of topological complexity for this space.
The first of these variations, the \emph{$r^{\text{th}}$-sequential topological complexity}, which we denote by $\textbf{TC}_{\textbf{r}}$, extends the motion planning problem of continuously finding a path between initial and final configurations to one of continuously finding a such a path that makes $r-2$ additional stops.
We show that

\begin{thm}\label{seq top complexity of conf(n,w)}
\thmtext
\end{thm}

Theorem \ref{seq top complexity of conf(n,w)} has the following real world interpretation: If one wants to write a program that governs the motion of $n$ robots moving around a long aisle of width $w$, then as long as $n>w$, such a program must consist of at least $r\big(n-\lceil\frac{n}{w}\rceil\big)$ different cases.
Though this result builds on previous result of the author concerning the classical topological complexity of this space \cite[Theorem 1.1]{wawrykow2024topological}, it uses entirely different methods.
Instead of laboriously studying the cohomology ring of $\text{conf}(n,w)$, we find a pair of homologically decomposable disjoint tori in $\text{conf}(n,w)$ \`{a} la Knudsen \cite{knudsen2024farber}.
We show that in general such tori not only yield a lower bound for $\textbf{TC}_{\textbf{r}}$, but also for $\textbf{dTC}_{r}$, the \emph{$r^{\text{th}}$-sequential distributional topological complexity}.
This variation of topological complexity cares not for minimizing the number of open sets covering $X^{r}$ with sections of the path projection map $\pi_{r}:X^{I}\to X^{r}$; instead, it only cares about minimizing how many such sets cover any point in $X^{r}$.
With this in mind, we have 

\begin{thm}\label{seq dist top complexity of conf(n,w)}
\thmtextone
\end{thm}

While our proof of Theorem \ref{seq top complexity of conf(n,w)} proves the third case of Theorem \ref{seq dist top complexity of conf(n,w)}, it does not prove the second---note, the first is trivial.
As such, we compute the $r^{\text{th}}$-sequential distributive topological complexity of $F_{n}(\R^{2})$, the ordered configuration space of points in the plane.
To do so, we generalize Knudsen's results about homologically decomposable disjoint tori \cite[Proposition 2.2]{knudsen2024farber}, yielding

\begin{cor}\label{seq dist top complexity of FnR2}
\thmtexttwo
\end{cor}

\subsection{Acknowledgements}
The author would like to thank Jes\'{u}s Gonz\'{a}lez for the suggestion to consider sequential topological complexity of $\text{conf}(n,w)$.
The author would also like to thank Alexander Dranishnikov and an anonymous referee for comments on an earlier draft of this paper.

\section{Topological Complexity}
Recall that topological complexity is a measure of the difficulty of the motion planning problem for a space.
Namely, if $X^{I}$ denotes the path space of $X$ with the compact-open topology, and
\begin{align*}
\pi:X^{I}&\to X\times X\\
\pi(\gamma)&=\big(\gamma(0), \gamma(1)\big)
\end{align*}
sends a path $\gamma$ to its endpoints, then the \emph{topological complexity of X} is a numerical homotopy invariant that counts the minimal number of open sets $U_{i}$ needed to cover $X\times X$ such that there exists a continuous section $s_{i}:U_{i}\to X^{I}$ of $\pi$ on each $U_{i}$.
We write $\textbf{TC}(X)$ for the topological complexity of $X$, and we set $\textbf{TC}(X):=\infty$ if there is no such finite cover $\{U_{i}\}$ of $X\times X$.
Note, we choose to work with a reduced version of topological complexity, which is Farber's original version shifted down by $1$.

\begin{exam}\label{tc points in Rn}
If we shrink our robots from disks to points, we move from $\text{conf}(n,w)$ to $F_{n}(\R^{2})$, the ordered configuration space of $n$ points in the plane.
When $n\ge 2$, Farber and Yuzvinsky proved that $\textbf{TC}\big(F_{n}(\R^{2})\big)=2n-3$ \cite[Theorem 1]{farber2002topological}, and Farber and Grant extended this to ordered configurations in arbitrary $\R^{m}$, proving that $\textbf{TC}\big(F_{n}(\R^{m})\big)=2n-3$ if $m$ is even and $\textbf{TC}\big(F_{n}(\R^{m})\big)=2n-2$ if $m$ is odd \cite[Theorem 1]{farber2009topological}.
\end{exam}

\begin{exam}
If $n\le w$, we have that $\text{conf}(n,w)$ and $F_{n}(\R^{2})$ are homotopic, so $\textbf{TC}\big(\text{conf}(n,w)\big)=2n-3$.
For $n>w$, these spaces are homotopically distinct, and $\textbf{TC}\big(\text{conf}(n,w)\big)=2\big(n-\lceil\frac{n}{w}\rceil\big)$ \cite[Theorem 1.1]{wawrykow2024topological}.
\end{exam}

There are several variations of topological complexity defined in the literature; we focus on two that correspond to paths that make several stops.

\subsection{Sequential Topological Complexity}

One can generalize the motion planning problem by asking for a continuous motion planner that takes in not only pair of points in $X$, but an $r$-tuple of points $x_{1},\dots, x_{r}$, and asks for a path from $x_{1}$ to $x_{2}$ to $x_{3}$, etc., that varies continuously in the choices of the $r$ points.
This can be formalized by writing
\begin{align*}
\pi_{r}:X^{I}&\to X^{r}\\
\pi_{r}(\gamma)&=\big(\gamma(0), \gamma\Big(\frac{1}{r-1}\Big),\dots, \gamma\Big(\frac{r-2}{r-1}\Big), \gamma(1)\big)
\end{align*}
for the fibration that associates a path $\gamma\in X^{I}$ to its values on the points $0, \frac{1}{r-1},\dots, \frac{r-2}{r-1}$, and $1$.
In this case, the motion planning problem becomes one of finding the minimal $k$ such that there are $k+1$ open sets $U_{0}, \dots, U_{k}$ covering $X^{r}$ such that that $\pi_{r}$ has a continuous section on each $U_{i}$.
This minimal $k$ is the \emph{$r^{\text{th}}$-sequential topological complexity of $X$}, and is denoted $\textbf{TC}_{\textbf{r}}$.
Note that if $r=2$, then $\textbf{TC}_{\textbf{2}}(X)=\textbf{TC}(X)$ as in this case the path makes no intermediate stops.
For the original formulation of sequential topological complexity see \cite{rudyak2010higher}.

\begin{exam}
Gonz\'{a}lez and Grant calculated $\textbf{TC}_{\textbf{r}}\big(F_{n}(\R^{m})\big)$ for all $n$ at least $2$ and all $m$, proving that $\textbf{TC}_{\textbf{r}}\big(F_{n}(\R^{m})\big)=r(n-1)-1$ if $m$ is even and $\textbf{TC}_{\textbf{r}}\big(F_{n}(\R^{m})\big)=r(n-1)$ if $m$ is odd \cite[Theorem 1.3]{gonzalez2015sequential}. 
Note that this agrees with, and builds on, the topological complexity results of Farber and Yuzvinsky and Farber and Grant described in Example \ref{tc points in Rn}.
\end{exam}

\begin{exam}
In the case of a graph $\Gamma$, Knudsen proved that $\textbf{TC}_{\textbf{r}}\big(F_{n}(\Gamma)\big)$ depends only on $r$ and the number of vertices of valence at least $3$ in $\Gamma$, once $n$ is large enough, proving a conjecture of Farber \cite[Theorem 1.1]{knudsen2022topological}.
\end{exam}

Directly computing the sequential topological complexity of a space is hard; fortunately, there exist calculable upper and lower bounds.
In our case these bounds will coincide, yielding $\textbf{TC}_{\textbf{r}}\big(\text{conf}(n,w)\big)$.

\begin{prop}\label{upper bound for seq top}
(Basabe--Gonz\'{a}lez--Rudyak--Tamaki \cite[Theorem 3.9]{basabe2014higher})
For any path-connected space $X$
\[
\textbf{TC}_{\textbf{r}}(X)\le r\frac{\text{hdim}(X)}{\text{conn}(X)+1},
\]
where $\text{hdim}(X)$ is the homotopy dimension of $X$ and $\text{conn(X)}$ is the connectivity of $X$.
\end{prop}

Lower bounds for (sequential) topological complexity are slightly harder to come by. 
Rudyak gives a lower bound for $\textbf{TC}_{\textbf{r}}(X)$ in terms of classes in $H^{*}(X^{r};R^{\otimes r})$ \cite[Proposition 3.4]{rudyak2010higher}.
Namely, Rudyak proves that if one can find $l$ distinct \emph{$r$-fold zero-divisors} in $H^{*}(X^{n};R^{\otimes r})$, i.e., classes that restrict to $0$ in $H^{*}(X;R)$ when pulled back along the diagonal inclusion $\Delta: X\to X^{r}$, whose product is non-zero in $H^{*}(X^{n};R^{\otimes r})$, then $\textbf{TC}_{r}(X)\ge l$.
One problem with this lower bound is that it relies on one having a strong understanding of the cohomology ring of $X$; in the case $X=\text{conf}(n,w)$, such an understanding is elusive, so we seek different methods to bound sequential topological complexity.
Fortunately, a more tractable lower bound for $\textbf{TC}_{\textbf{r}}(X)$ can be obtained by finding ``nice'' tori in $X$. 

An \emph{$m$-torus in $X$} is a map $f:T^{m}\to X$, and we say that an $m$-torus in $X$ is \emph{homologically decomposable} if the homomorphism induced on $H_{1}$ is injective.
Additionally, we say that a pair of tori in $X$ are \emph{homologically disjoint} if the images of the two homomorphisms induced on $H_{1}$ intersect trivially.
Knudsen proved that if there exist two homologically disjoint decomposable $m$-tori in $X$, then $rm\le \textbf{TC}_{\textbf{r}}(X)$ \cite[Proposition 2.2]{knudsen2024farber}; we extend his argument to homologically decomposable disjoint tori of different dimensions.

\begin{prop}\label{disjoint tori}
If $X$ admits a homologically decomposable $m$-torus and a homologically decomposable $l$-torus that are homologically disjoint, then
\[
\textbf{TC}_{\textbf{r}}(X)\ge (r-1)m+l
\]
for all $r\ge 2$.
\end{prop}

Our proof of Proposition \ref{disjoint tori} follows Knudsen's proof of \cite[Proposition 2.2]{knudsen2024farber}; we point the interested reader to that paper for a further discussion of the ideas motivating these results.

\begin{proof}
Let $f$ be the decomposable $m$-torus and $g$ the decomposable $l$-torus.
Since these tori are decomposable, the universal coefficients theorem implies that the induced homorphisms $f^{*}:H^{1}(X)\to H^{1}(T^{m})$ and $g^{*}:H^{1}(X)\to H^{1}(T^{l})$ are surjective.
Given that we can write $H^{1}(T^{m})\cong \bigwedge [x_{1}, \dots, x_{m}]$ and $H^{1}(T^{l})\cong \bigwedge [x_{1}, \dots, x_{l}]$, let $y_{i}$ be the preimage of $x_{i}$ in $H^{1}(X)$ induced by $f$ and $z_{i}$ be the preimage of $x_{i}$ in $H^{1}(X)$ induced by $g$.
Given a class $a\in H^{*}(X)$, consider the class $\zeta(a):=a\otimes 1-1\otimes a\in H^{*}(X\times X)$.
For $1\le i< j\le r$, we write $\zeta^{ij}(a)$ for the pullback of $\zeta(a)$ along the projection $X^{r}\to X\times X$ onto the $i^{\text{th}}$ and $j^{\text{th}}$ factors.
Setting $\zeta^{ij}_{f}:=\prod^{m}_{k=1}\zeta^{ij}(y_{k})$ and $\zeta^{ij}_{g}:=\prod^{l}_{k=1}\zeta^{ij}(z_{k})$ one can check that
\[
\zeta^{12}_{g}\zeta^{12}_{f}\prod^{r}_{k=3}\zeta^{(k-1)k}_{f}=\sum_{S_{p}\in\{1, \dots, m\}, 1\le k\le r}\pm z_{S_{1}}y_{S_{2}}\otimes z_{S_{1}^{c}}y_{S_{2}^{c}}y_{S_{3}}\otimes y_{S^{c}_{3}}y_{S_{4}}\otimes\cdots\otimes y_{S^{c}_{r-1}}y_{S_{r}}\otimes y_{S_{r}^{c}}.
\]
We claim that this product of $m(r-1)+l$ total $r$-fold zero-divisors is non-trivial.

To see that this class is non-zero, we apply $(f^{r-1}\times g)^{*}$ to it to get a class in $H^{m(r-1)+l}(T^{m(r-1)+l})$ and evaluate it on the fundamental class of $T^{m(r-1)+l}$.
One can check that any term with $S_{r}\neq \emptyset$ vanishes for degree reasons; continuing in this manner shows that the only classes that can contribute non-trivially are those such that $S_{k}=\emptyset$ for $k>2$.
Furthermore, we must have that $S_{2}=S_{1}^{c}=\emptyset$, so the only term that matters is
\[
(f^{r-1}\times g)^{*}\Big(\big(\prod^{m}_{k=1}y_{k}\big)^{\otimes(r-1)}\otimes \prod_{k=1}^{l}z_{k}\Big)
\]
Since $f$ and $g$ are homologically disjoint, it follows that we can assume that $y_{i}$, resp. $z_{i}$, evaluates to $0$ on the image of $g_{*}$, resp. $f_{*}$, so this is equal
\[
\Big(\prod^{m}_{k=1}x_{i}\Big)^{\otimes (r-1)},
\]
which evaluates to $1$ on the fundamental class of $T^{m(r-1)+l}$.
Therefore, the product $\zeta^{12}_{g}\zeta^{12}_{f}\prod^{r}_{k=3}\zeta^{(k-1)k}_{f}$ is non-trivial, yielding the bound $\textbf{TC}_{\textbf{r}}(X)\ge (r-1)m+l$.
\end{proof}

A careful examination of the proof Proposition \ref{disjoint tori} and an application of the Universal Coefficients Theorem shows that these tori yield rational $r$-fold zero-divisiors.
We will this fact to bound the sequential distributional topological complexity of $\text{conf}(n,w)$, a variation of topological complexity whose definition we recall in the next subsection. 

\subsection{Distributional Topological Complexity}

If the sequential topological complexity of $X$ is finite, then there is some covering $\{U_{i}\}$ of $X^{r}$ such that there are sections $s_{i}:U_{i}\to X^{I}$ of $\pi_{r}$.
In determining $\textbf{TC}_{\textbf{r}}(X)$, our only concern is that the number of $U_{i}$ be as small as possible; we do not care to which path in $X^{I}$ an $(x_{1}, \dots, x_{r})\in U_{i}\cap U_{j}$ gets sent.
This digital method of path-choosing has its upsides, though one could argue that it does not fully capture the complexity of the motion planning problem as it might be the case that at every point of $X^{r}$ one might need to decide between only two choices of path, not $\textbf{TC}_{r}(X)$ of them, e.g., $\R\mathbb{P}^{d}$ for $d>0$ \cite[Proposition 6.9]{knudsen2024analog}.
One way of formalizing this different method of measuring the difficultly of the motion planning problem for a metric space is to consider $\mathcal{B}_{n}(X^{I})$ the space of probability measures on the path space of $X$ whose support has at most $n$ elements.
The \emph{$r^{\text{th}}$-sequential distributional topological complexity of $X$}, denoted $\textbf{dTC}_{\textbf{r}}(X)$, is the minimal $n$ such that there is a map
\[
s_{r}:X^{r}\to \mathcal{B}_{n}(X^{I})
\]
that sends $(x_{1}, \dots, x_{r})$ to $\mathcal{B}_{n}\big((x_{1}, \dots, x_{r})^{I}\big)\subset\mathcal{B}_{n}(X^{I})$.
While this variation of topological complexity might be difficult to grasp on first glance, we note that it has a relatively simple interpretation: $\textbf{dTC}_{\textbf{r}}(X)$ is the minimal number $n$ such that each $(x_{1}, \dots, x_{r})\in X^{r}$ is contained in at most $n$ sets $U_{i}$, where $\{U_{i}\}$ is a finite open cover of $X^{r}$ with a continuous section $s_{i}$ of $\pi_{r}$ on each $U_{i}$.
For the reader more interested in distributional topological complexity we suggest \cite{jauhari2025sequential, dranishnikov2024distributional} as well as \cite{knudsen2024analog}, which talks about the closely related notion of analog topological complexity.

Like classical topological complexity, (sequential) distributional topological complexity is hard to calculate directly.
Fortunately, there are upper and lower bounds, the former of which is relatively easy to come by.

\begin{prop}\label{dtcm is less than tcm}
(Jauhari \cite[Proposition 3.4]{jauhari2025sequential})
For each $r\ge 2$,
\[
\textbf{dTC}_{\textbf{r}}(X)\le \textbf{TC}_{\textbf{r}}(X).
\]
\end{prop}

Lower bounds for $\textbf{dTC}_{\textbf{r}}(X)$ are a bit more harder to find, though we will be able to adapt the following result of Jauhari for our purposes.

\begin{lem}\label{Jauhari lower bound}
(Jauhari  \cite[Theorem 4.4]{jauhari2025sequential}) Suppose that $\alpha^{*}_{i}\in H^{k_{i}}\big(SP^{n!}(X^{r});R)$ for $1\le i\le n$, some ring $R$, and $k_{i}\ge 1$, are cohomology classes such that $\Delta^{*}_{n}(\alpha^{*}_{i})=0$, where $\Delta_{n}:SP^{n!}(X)\to SP^{n!}(X^{r})$ is induced by the standard inclusion of $X$ into $X^{r}$ along the diagonal.
Let $\alpha_{i}$ be their images in $H^{k_{i}}(X^{r};R)$ induced by the diagonal inclusion of $X^{r}$ into $SP^{n!}(X^{r})$, and assume that the cup product $\alpha_{1}\smile \alpha_{2}\smile \cdots \smile \alpha_{n}\neq 0\in H^{*}\big(X^{r};R\big)$.
Then 
\[
\textbf{dTC}_{\textbf{r}}(X)\ge n.
\]
\end{lem}

This result is stated in terms of Alexander--Spanier cohomology if $X$ is a locally finite CW complex this coincides with singular cohomology, and Lemma \ref{Jauhari lower bound} can be modified to yield a simpler lower bound for $dTC_{m}(X)$.

\begin{cor}\label{rational zero-divisor lower bound for dtcm}
Let $X$ have the homotopy type of a finite simplicial complex. 
The rational $r$-fold zero-divisor cup-length of $X^{r}$ is a lower bound for $\textbf{dTC}_{\textbf{r}}(X)$.
\end{cor}

\begin{proof}
This follows immediately from Lemma \ref{Jauhari lower bound} and the fact that for any finite simplicial complex $Y$ and any $l\ge 1$, the inclusion $\delta_{l}:Y\to SP^{l}(Y)$ that sends $y$ to $(y, \dots, y)$ induces a surjection from $H^{*}\big(SP^{l}(Y);\Q\big)$ to  $H^{*}(Y;\Q)$ \cite[Proposition 4.3]{dranishnikov2024distributional}.
In particular, this holds for $Y=X^{r}$.
\end{proof}

Since the proof of Proposition \ref{disjoint tori} \cite[Proposition 2.2]{knudsen2024farber} shows that if can find a homologically decomposable $m$-torus in $X$ that is homologically disjoint from a homologically decomposable $l$-torus in $X$, then one can find a non-zero product of $r(m-1)+l$ total $r$-fold zero-divisors in $H^{1}(X^{r};\Q)$, which yields the following corollary: 

\begin{cor}\label{tori yield dtc lower bound}
Let $X$ have the homotopy type of a finite simplicial complex.
If $X$ admits a homologically decomposable $m$-torus and a homologically decomposable $l$-torus that are homologically disjoint, then 
\[
\textbf{dTC}_{\textbf{r}}(X)\ge r(m-1)+l
\]
for all $r>1$.
\end{cor}

Given Proposition \ref{disjoint tori} and Corollary \ref{tori yield dtc lower bound}, we seek a disjoint pair homologically decomposable tori in $\text{conf}(n,w)$ to bound its sequential (distributional) topological complexity.

\section{$\text{conf}(n,w)$}

We begin our calculation of $(\textbf{d})\textbf{TC}_{\textbf{r}}\big(\text{conf}(n,w)\big)$ by recalling results of Alpert, Kahle, and MacPherson \cite{alpert2021configuration} and Alpert and Manin \cite{alpert2021configuration1}.
These result will allow us to compute the homotopy dimension of $\text{conf}(n,w)$, giving us an upper bound for $\textbf{TC}_{\textbf{r}}\big(\text{conf}(n,w)\big)$ via Proposition \ref{upper bound for seq top}.
They also point toward the existence of large homologically decomposable disjoint tori, which, via Proposition \ref{disjoint tori} and Corollary \ref{tori yield dtc lower bound}, yields a lower bound for $(\textbf{d})\textbf{TC}_{\textbf{r}}\big(\text{conf}(n,w)\big)$.
We will spend the majority of our time investigating these tori.

While $\text{conf}(n,w)$ is relatively easy to conceptualize---for an example of a point in $\text{conf}(n,w)$, see Figure \ref{conf52}---it is not easy to work with.
As such, we recall a finite cellular complex introduced by Blagojevi\'{c} and Ziegler in \cite{blagojevic2014convex} called $\text{cell}(n,w)$.
Alpert, Kahle, and MacPherson proved that this complex is homotopy equivalent to $\text{conf}(n,w)$ \cite[Theorem 3.1]{alpert2021configuration}), and we will use it to get an upper bound on the sequential topological complexity of $\text{conf}(n,w)$.

\begin{figure}[h]
\centering
\captionsetup{width=.8\linewidth}
\includegraphics[width = 12cm]{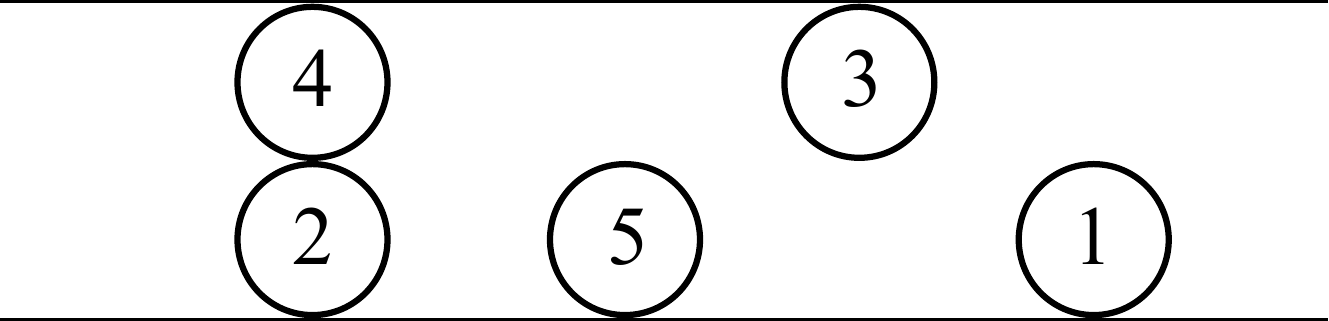}
\caption{A point in $\text{conf}(5,2)$.
Note that two disks can be aligned vertically, but three cannot.
}
\label{conf52}
\end{figure}

Let \emph{$\text{cell}(n)$} denote the cellular complex whose cells are represented by \emph{symbols} consisting of an ordering of the numbers $1, \dots, n$ separated by vertical bars into \emph{blocks} such that no block is empty. 
A cell $f\in \text{cell}(n)$ is a codimensional-$1$ \emph{face} of a cell $g\in \text{cell}(n)$ if $g$ can be obtained by deleting a bar in $f$ and shuffling the resulting block.
We say that $g$ is a \emph{co-face} of $f$.
It follows that cells represented by symbols with no bars are the top-dimensional cells of $\text{cell}(n)$.
It follows that we can view $\text{cell}(n)$ as an $(n-1)$-dimensional complex, and each bar in a symbol lowers the dimension of the corresponding cell by $1$.

By restricting how big a block can be in $\text{cell}(n)$, one gets the cellular complex $\emph{\text{cell}(n,w)}$.
This is the subcomplex of $\text{cell}(n)$ consisting of the cells represented by symbols whose blocks have at most $w$ elements; see Figure \ref{cell32} for an example.
The following result of Alpert, Kahle, and MacPherson allows us to study $\text{cell}(n,w)$ in place of $\text{conf}(n,w)$.

\begin{figure}[h]
\centering
\captionsetup{width=.8\linewidth}
\includegraphics[width = 8cm]{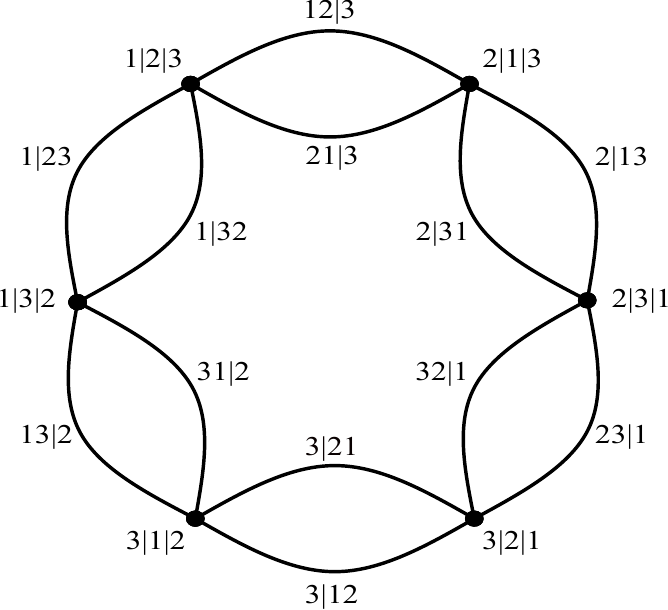}
\caption{The cellular complex $\text{cell}(3,2)$.
The loop corresponding the $1$-cells $1\,2|3$ and $2\,1|3$ corresponds to disks $1$ and $2$ orbiting each other while disk $3$ sits to the right of them.
See the right side of Figure \ref{disjointtoriconf32}.}
\label{cell32}
\end{figure}

\begin{prop}
(Alpert--Kahle--MacPherson \cite[Theorem 3.1]{alpert2021configuration}) There is an $S_{n}$-equivariant homotopy equivalence $\text{conf}(n,w)\simeq \text{cell}(n,w)$.
\end{prop}

One can quickly check that
\[
\dim\big(\text{cell}(n,w)\big)=n-\Big\lceil\frac{n}{w}\Big\rceil;
\]
for example, see \cite[Proposition 3.2]{wawrykow2024topological}.

Since $\text{conf}(n,w)$ is clearly connected, it follows from Propositions \ref{upper bound for seq top} and \ref{dtcm is less than tcm} that we have the following upper bounds for the $r^{\text{th}}$-sequential (distributional) topological complexity of $\text{conf}(n,w)$.

\begin{prop}\label{upper bound for seq top of conf}
\[
\textbf{dTC}_{\textbf{r}}\big(\text{conf}(n,w)\big)\le \textbf{TC}_{\textbf{r}}\big(\text{conf}(n,w)\big)\le r\Big(n-\Big\lceil\frac{n}{w}\Big\rceil\Big).
\]
\end{prop}

Having found an upper bound for $(\textbf{d})\textbf{TC}_{\textbf{r}}\big(\text{conf}(n,w)\big)$, we begin the more challenging task of finding strong lower bounds.
To approach this problem we recall Alpert and Manin's basis for the homology of $\text{conf}(n,w)$.

If $n\le w$, we have that $\text{conf}(n,w)\simeq F_{n}(\R^{2})$.
In this case a \emph{wheel} \emph{$W(i_{1}, \dots, i_{n})$} is an embedded $(n-1)$-torus in $\text{conf}(n,w)$ corresponding the disks $i_{1}$ and $i_{2}$ orbiting each other, the disk $i_{3}$ independently orbiting the disks $i_{1}$ and $i_{2}$, the disk $i_{4}$ independently orbiting $i_{1}$, $i_{2}$, and $i_{3}$, etc., e.g., see Figure \ref{W132}.

\begin{figure}[h]
\centering
\captionsetup{width=.8\linewidth}
\includegraphics[width = 8cm]{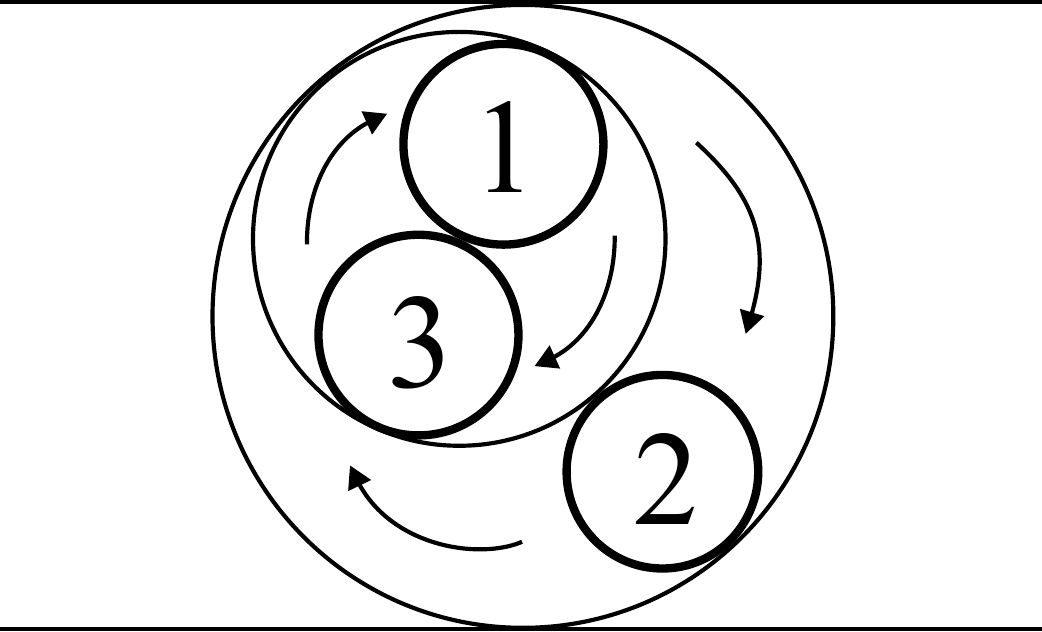}
\caption{The wheel $W(1,3,2)$ in $\text{conf}(3,3)$.
Disks $1$ and $3$ orbit each other, while disk $2$ independently orbits them.}
\label{W132}
\end{figure}

We will indiscriminately use the word wheel when talking about either the embedded torus or the corresponding class in homology.
Additionally, we can take \emph{concatenation products} of these wheels by embedding them left to right in the infinite strip.
Note, two wheels commute (up to sign) if and only if the total number of disks in the two wheels is at most $w$.
Moreover, the concatenation product of any number of wheels is a torus in $\text{conf}(n,w)$.

There is another operation on wheels in addition to taking concatenation products, that yields homology classes called \emph{(averaged)-filters}. 
These classes are what truly differentiate $\text{conf}(n,w)$ from $F_{n}(\R^{2})$ as when $w>2$, they preclude the former from being a $K(\pi,1)$, though we only recall their existence to state the following basis theorem.
First, however, we note that Alpert and Manin proved that $H_{*}\big(\text{conf}(n,w);\Z\big)$ is free Abelian \cite[Theorem B]{alpert2021configuration1}, and that concatenations of wheels representation integral homology classes.

\begin{prop}\label{AMWthmB''}
(W. \cite[Theorem 3.35]{wawrykow2023representation}) The homology group $H_{k}\big(\text{conf}(n,w);\Q\big)$ has a basis consisting of concatenations of wheels and non-trivial  averaged-filters on $m\ge 3$ wheels such that in each wheel, inside or outside of an averaged-filter, the largest label comes first.
We say one wheel ranks above another if it has more disks or if they have the same number of disks and its largest label is greater. 
A cycle is in the basis if and only if:
\begin{enumerate}
\item The wheels inside each  averaged-filter are in order of increasing rank. 
\item If $W_{1}$ and $W_{2}$ are adjacent wheels consisting of $n_{1}$ and $n_{2}$ disks, respectively, where $W_{1}$ is to the left of $W_{2}$, then $W_{1}$ outranks $W_{2}$ or $n_{1}+n_{2} > w$.
\item Every wheel to the immediate left of a  averaged-filter ranks above the least wheel in the averaged-filter.
\end{enumerate}
\end{prop}

For more on the homology of $\text{conf}(n,w)$ see \cite[Sections 3 and 4]{alpert2021configuration1} and \cite[Section 3]{wawrykow2023representation}.

One can check that the inclusion of a wheel into $\text{conf}(n,w)$ induces an inclusion on the homology 
\[
H_{*}(T^{n-1})\hookrightarrow H_{*}\big(\text{conf}(n,w)\big),
\]
making the wheels homologically decomposable.

\begin{prop}\label{image of wheel in H1}
Let $n\le w$, and let $W(i_{1}, \dots, i_{n})$ be a wheel in $\text{conf}(n,w)$.
The image of $W(i_{1}, \dots, i_{n})$ in $H_{1}\big(\text{conf}(n,w)\big)$ is
\[
W(i_{1}, i_{2})\oplus\big(W(i_{1}, i_{3})+W(i_{2}, i_{3})\big)\oplus\cdots\oplus\big(W(i_{1}, i_{n})+\cdots+W(i_{n-1}, i_{n})\big).
\]
Here we have suppressed the wheels on $1$ disk, which in light of Theorem \ref{AMWthmB''} appear in order of decreasing label to the right of each wheel on two disks.
\end{prop}

\begin{proof}
This follows immediately from the description of the wheels and the fact that the class in first homology corresponding to disk $i_{k}$ orbiting disks $i_{1}, \dots, i_{k-1}$ can be expressed as the sum
\[
W(i_{1}, i_{k})+\cdots+W(i_{k-1}, i_{k}).
\]
\end{proof}

Concatenation products of wheels behave similarly, and this will allow us to find large disjoint tori in $\text{conf}(n,w)$.

\begin{prop}\label{decomposing concatenation products of wheels}
Let $W(i_{1,1},\dots, i_{1, n_{1}})\cdots W(i_{m, 1}, \dots, i_{m, n_{m}})$ be a concatenation product of wheels.
Such an embedded torus is homologically decomposable, and if $w=2$, the image of this torus in $H_{1}\big(\text{conf}(n,w)\big)$ is 
\begin{multline*}
W(i_{1, 1}, i_{1, n_{1}})W(i_{2,1})W(i_{2, n_{2}})\cdots W(i_{m, n_{m}})\oplus\cdots\\\oplus
W(i_{1, 1})W(i_{1, n_{1}})\cdots W(i_{m-1,1})W(i_{m-1, n_{m-1}})\cdots W(i_{m,1}, i_{m, n_{m}}),
\end{multline*}
whereas if $w>2$, the image of this torus in $H_{1}\big(\text{conf}(n,w)\big)$ is 
\begin{multline*}
W(i_{1,1}, i_{1,2})\oplus\big(W(i_{1,1}, i_{1,3})+W(i_{1,2}, i_{1,3})\big)\oplus\cdots\oplus\big(W(i_{1,1}, i_{1,n_{1}})+\cdots+W(i_{1,n_{1}-1}, i_{1,n_{1}})\big)\oplus\cdots\\
\oplus W(i_{m,1}, i_{m,2})\oplus\big(W(i_{m,1}, i_{m,3})+W(i_{m,2}, i_{m,3})\big)\oplus\cdots\oplus\big(W(i_{m,1}, i_{m,n_{m}})+\cdots+W(i_{m,n_{m}-1}, i_{m,n_{m}})\big),
\end{multline*}
where, in the second decomposition, we have suppressed the wheels on one disk.
\end{prop}

\begin{proof}
This follows immediately from Proposition \ref{image of wheel in H1} and the definition of the concatenation product.
\end{proof}

Now we demonstrate the existence of two homologically decomposable disjoint $\Big(n-\big\lceil\frac{n}{w}\big\rceil\Big)$-tori in $\text{conf}(n,w)$. 
This, along with Proposition \ref{disjoint tori} and Corollary \ref{tori yield dtc lower bound} will prove that $(\textbf{d})\textbf{TC}_{\textbf{r}}\big(\text{conf}(n,w)\big)$ is at least $r\Big(n-\big\lceil\frac{n}{w}\big\rceil\Big)$.
We consider the cases $w=2$ and $w>2$ separately due to the fact that wheels on 2 disks in $\text{conf}(n,2)$ split the strip into disjoint regions.

\begin{lem}\label{disjoint tori in conf(n,2)}
There are two homologically decomposable disjoint $\big(n-\big\lceil\frac{n}{2}\big\rceil\big)$-tori in $\text{conf}(n,2)$.
\end{lem}

\begin{proof}
Consider the following wheels in $\text{conf}(n,2)$:
\[
A=
\begin{cases}
W\big(n, \frac{n}{2}\big)W\big(n-1, \frac{n}{2}-1\big)\cdots W\big(\frac{n}{2}+1, 1\big)&n\mbox{ even}\\
W\big(n, \frac{n-1}{2}\big)W\big(n-1, \frac{n-3}{2}\big)\cdots W\big(\frac{n+3}{2}, 1\big)W\big(\frac{n+1}{2}\big)&n\mbox{ odd}
\end{cases}
\]
and 
\[
B=
\begin{cases}
W(n, 1)W\big(n-1, \frac{n}{2}\big)W\big(n-2, \frac{n}{2}-1\big)\cdots W\big(\frac{n}{2}+1, 2\big) &n\mbox{ even}\\
W\big(n-1, \frac{n-1}{2}\big)W\big(n-2, \frac{n-3}{2}\big)\cdots W\big(\frac{n+1}{2}, 1\big)W(n)&n\mbox{ odd},
\end{cases}
\]
see Figures \ref{disjointtoriconf32} and \ref{disjointtoricell32}.

\begin{figure}[h]
\centering
\captionsetup{width=.8\linewidth}
\includegraphics[width = 12cm]{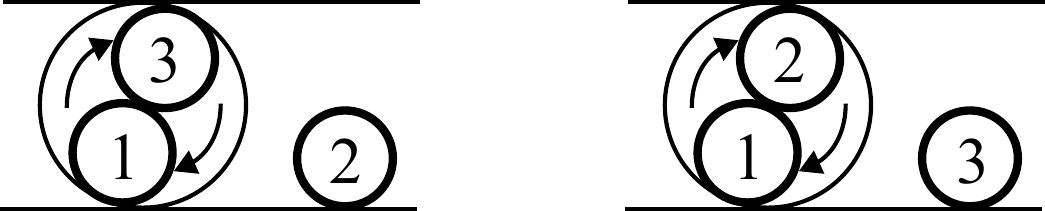}
\caption{The tori $A=W(3,1)W(2)$ and $B=W(2,1)W(3)$ in $\text{conf}(3,2)$.
This is a disjoint pair of homologically decomposable tori, see Figure \ref{disjointtoricell32}.}
\label{disjointtoriconf32}
\end{figure}

\begin{figure}
\centering
\captionsetup{width=.8\linewidth}
\includegraphics[width = 8cm]{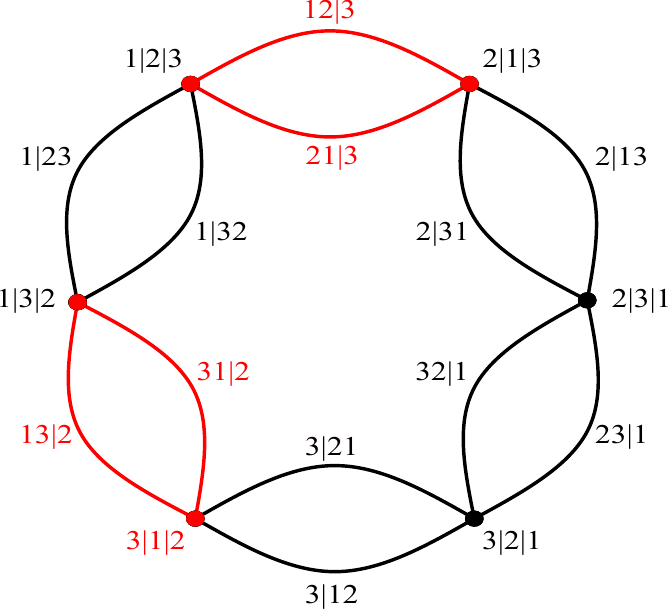}
\caption{The tori $A=W(3,1)W(2)$ and $B=W(2,1)W(3)$ in $\text{cell}(3,2)$.
In the $\text{cell}(3,2)$ model of $\text{conf}(3,2)$ it is obvious that these tori are homologically disjoint. 
Compare with Figure \ref{disjointtoriconf32}.}
\label{disjointtoricell32}
\end{figure}

These wheels are homologically decomposable, and when $n$ is even Proposition \ref{decomposing concatenation products of wheels} proves that their images in $H_{1}\big(\text{conf}(n,2)\big)$ are
\begin{multline*}
\text{Im}(A)=W\Big(n, \frac{n}{2}\Big)W(n-1)\cdots\widehat{W\Big(\frac{n}{2}\Big)}W(1)\oplus\\
W(n)W\Big(\frac{n}{2}\Big)W\Big(n-1, \frac{n}{2}-1\Big)W(n-2)\cdots\widehat{W\Big(\frac{n}{2}-1\Big)}\widehat{W\Big(\frac{n}{2}\Big)}\cdots W(1)\oplus\\
\cdots\oplus W(n)\cdots\widehat{W\Big(\frac{n}{2}+1\Big)}\cdots W(2)W\Big(\frac{n}{2}+1, 1\Big),
\end{multline*}
and
\begin{multline*}
\text{Im}(B)=W(n,1)W(n-1)\cdots W(2)\oplus\\
W(n)W(1)W\Big(n-1, \frac{n}{2}\Big)W(n-2)\cdots\widehat{W\Big(\frac{n}{2}\Big)}\cdots W(2)\oplus\\
\cdots\oplus W(n)\cdots\widehat{W\Big(\frac{n}{2}+1, 2\Big)}\cdots\widehat{W(2)}W(1)W\Big(\frac{n}{2}+1, 2\Big).
\end{multline*}

Similarly, when $n$ is odd Proposition \ref{decomposing concatenation products of wheels} proves their images in $H_{1}\big(\text{conf}(n,2)\big)$ are
\begin{multline*}
\text{Im}(A)=W\Big(n, \frac{n-1}{2}\Big)W(n-1)\cdots\widehat{W\Big(\frac{n-1}{2}\Big)}\cdots W(1)\oplus\\
W(n)W\Big(\frac{n-1}{2}\Big)W\Big(n-1, \frac{n-3}{2}\Big) W(n-2)\cdots \widehat{W\Big(\frac{n-1}{2}\Big)}\widehat{W\Big(\frac{n-3}{2}\Big)}\cdots W(1)\oplus\\
\cdots\oplus W(n)\cdots\widehat{W\Big(\frac{n+3}{2}\Big)}\widehat{W\Big(\frac{n+1}{2}\Big)}\cdots W(2)W\Big(\frac{n+3}{2}, 1\Big)W\Big(\frac{n+1}{2}\Big),
\end{multline*}
and
\begin{multline*}
\text{Im}(B)=W\Big(n-1, \frac{n-1}{2}\Big)W(n)\widehat{W(n-1)}W(n-2)\cdots\widehat{W\Big(\frac{n-1}{2}\Big)}\cdots W(1)\oplus\\
W(n-1)W\Big(\frac{n-1}{2}\Big)W\Big(n-2, \frac{n-3}{2}\Big)W(n)\widehat{W(n-1)}\widehat{W(n-2)}W(n-3)\cdots \widehat{W\Big(\frac{n-1}{2}\Big)}\widehat{W\Big(\frac{n-3}{2}\Big)}\cdots W(1)\oplus\\
\cdots\oplus W(n-1)\cdots\widehat{W\Big(\frac{n+1}{2}\Big)}\cdots W(2)W\Big(\frac{n+1}{2}, 1\Big)W(n).
\end{multline*}

In both cases, we used the fact that we may commute the wheels on $1$ disk past each other to rewrite our summands.
Moreover, in both cases, all of the summands of $\text{Im}(A)$ and $\text{Im}(B)$ in $H_{1}\big(\text{conf}(n,2)\big)$ are elements of the basis of Proposition \ref{AMWthmB''}. 
Moreover, the basis elements spanning $\text{Im}(A)$ are disjoint from the basis elements spanning $\text{Im}(B)$.
It follows that $A$ and $B$ are homologically disjoint in $\text{conf}(n,2)$.
\end{proof}

\begin{lem}\label{disjoint tori in conf(n,w) w>2}
There are two homologically decomposable disjoint $\big(n-\big\lceil\frac{n}{w}\big\rceil\big)$-tori in $\text{conf}(n,w)$.
\end{lem}

\begin{proof}
We have already proven this when $w=2$ in Lemma \ref{disjoint tori in conf(n,2)}, so we focus on the case $w>2$.
Let $m=\big\lceil\frac{n}{w}\big\rceil$.
Consider the following products of wheels in $\text{conf}(n,w)$
\begin{multline*}
A=W\big(n, n-m, n-m-1,\dots, n-m-w+2\big)W\big(n-1, n-m-w+1, \dots, n-m-2(w-2)-1\big)\cdots\\
W\big(n-m+2, n-m-(m-2)(w-2)-2,\dots, n-m-(m-1)(w-2)-m+2\big)W\big(n-m+1, n-(m-1)w-1, \dots, 1\big)
\end{multline*}
and
\begin{multline*}
B=W\big(n-1, n-m, n-m-1,\dots, n-m-w+2\big)\cdots W\big(n-m+2, n-m-w+1, \dots, n-m-2(w-2)-1\big)\\
W\big(n-m+1, n-m-(m-2)(w-2)-2,\dots, n-m-(m-1)(w-2)-m+2\big)W\big(n, n-(m-1)w-1, \dots, 1\big);
\end{multline*}
see Figure \ref{disjointtoriconf73} for an example.

\begin{figure}[h]
\centering
\captionsetup{width=.8\linewidth}
\includegraphics[width = 8cm]{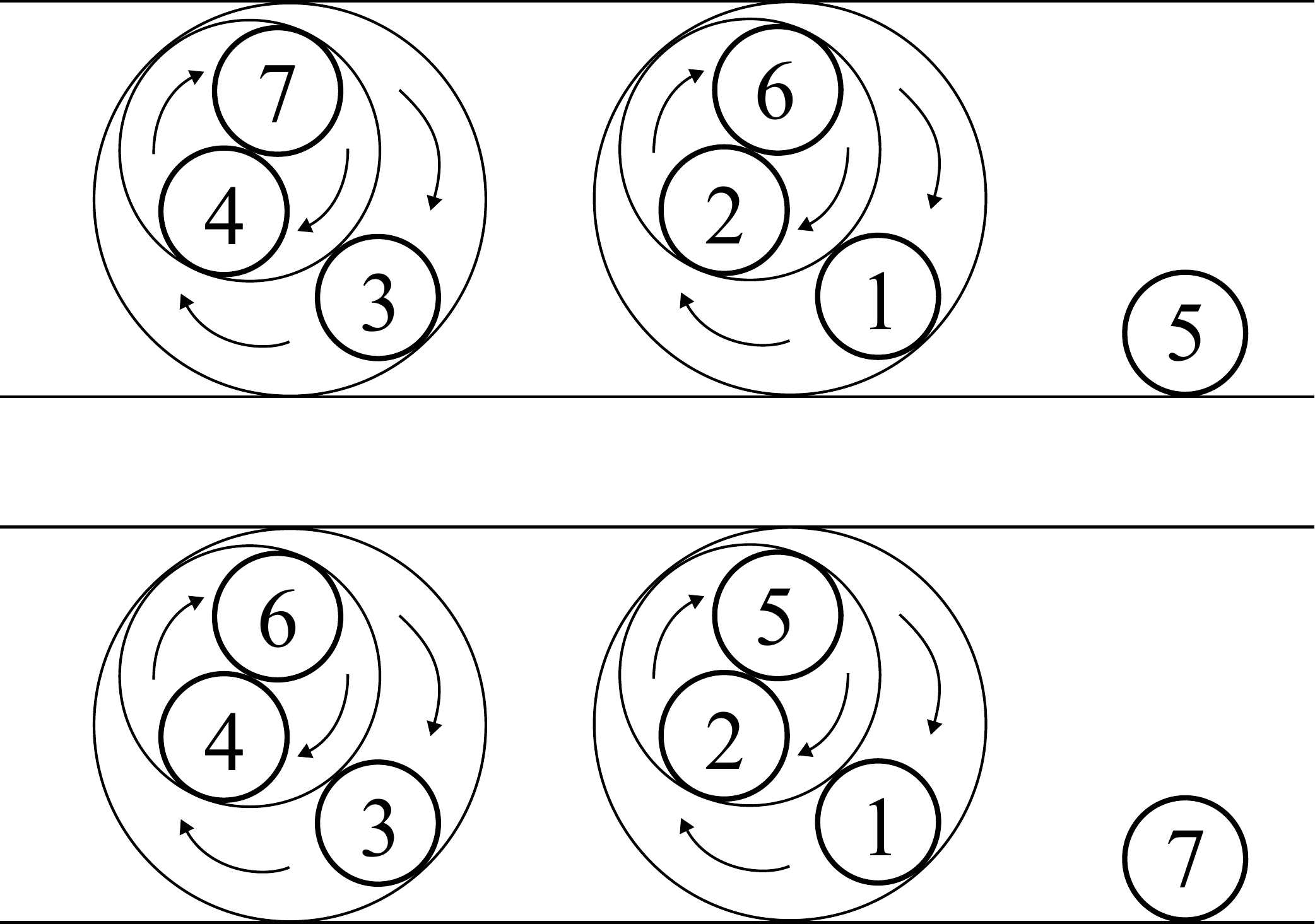}
\caption{The tori $A=W(7,4,3)W(6,2,1)W(5)$ and $B=W(6,4,3)W(5,2,1)W(4)$ in $\text{conf}(7,3)$.
This is a disjoint pair of homologically decomposable tori.}
\label{disjointtoriconf73}
\end{figure}

By Proposition \ref{decomposing concatenation products of wheels} these products of wheels are decomposable, and their images in $H_{1}\big(\text{conf}(n,w)\big)$, when written in terms of the basis of Proposition \ref{AMWthmB''}, are
\begin{multline*}
\text{Im}(A)=W(n, n-m)\oplus\big(W(n,n-m-1)+W(n-m,n-m-1)\big)\oplus\cdots\\
\oplus\big(W(n,n-m-w+2)+\cdots+W(n-m-w+3,n-m-w+2)\big)\bigoplus\\
W(n-1, n-m-w+1)\oplus\big(W(n-1, n-m-w)+W(n-m-w, n-m-w+1)\big)\oplus\cdots\\
\oplus\big(W(n-1,n-m-2(w-2)-1)+\cdots+W(n-m-2(w-2),n-m-2(w-2)-1)\big)\bigoplus\cdots\\
\bigoplus W(n-m+1, n-(m-1)w-1)\oplus\big(W(n-m+1, n-(m-1)w-2)+W(n-(m-1)w-1, n-(m-1)w-2)\big)\oplus\cdots\\
\oplus\big(W(n-m+1,1)+W(n-(m-1)w-2, 1)+\cdots+W(2,1)\big),
\end{multline*}
and
\begin{multline*}
\text{Im}(B)=W(n-1, n-m)\oplus\big(W(n-1,n-m-1)+W(n-m,n-m-1)\big)\oplus\cdots\\
\oplus\big(W(n-1,n-m-w+2)+\cdots+W(n-m-w+3,n-m-w+2)\big)\bigoplus\\
W(n-2, n-m-w+1)\oplus\big(W(n-2, n-m-w)+W(n-m-w, n-m-w+1)\big)\oplus\cdots\\
\oplus\big(W(n-2,n-m-2(w-2)-1)+\cdots+W(n-m-2(w-2),n-m-2(w-2)-1)\big)\bigoplus\cdots\\
\bigoplus W(n, n-(m-1)w-1)\oplus\big(W(n, n-(m-1)w-2)+W(n-(m-1)w-1, n-(m-1)w-2)\big)\oplus\cdots\\
\oplus\big(W(n,1)+W(n-(m-1)w-2, 1)+\cdots+W(2,1)\big).
\end{multline*}

There are $\big(n-\big\lceil\frac{n}{w}\big\rceil\big)$-dimensional subspaces of $H_{1}\big(\text{conf}(n,w)\big)$.
We project them onto the $\big(n-\big\lceil\frac{n}{w}\big\rceil\big)$-dimensional subspace
\begin{multline*}
W(n, n-m)\oplus W(n,n-m-1)\oplus\cdots \oplus W(n,n-m-w+2)\bigoplus\\
W(n-1, n-m-w+1)\oplus W(n-1, n-m-w)\oplus\cdots\oplus W(n-1,n-m-2(w-2)-1)\bigoplus\cdots\\
\bigoplus W(n-m+1, n-(m-1)w-1)\oplus W(n-m+1, n-(m-1)w-2)\oplus\cdots\oplus W(n-m+1, 1),
\end{multline*}
which is also an $\big(n-\big\lceil\frac{n}{w}\big\rceil\big)$-dimensional subspaces of $H_{1}\big(\text{conf}(n,w)\big)$.

The image of the projection of $A$ is all of this subspace, whereas the image of $B$ is trivial.
Since the dimension of $A$ and this subspace are the same, it follows that the images of $A$ and $B$ trivially intersect in $H_{1}\big(\text{conf}(n,w)\big)$, i.e., they are homologically disjoint.
\end{proof}

The previous two lemmas, Proposition \ref{disjoint tori}, and Corollary \ref{tori yield dtc lower bound} yield the following lower bound for the $r^{\text{th}}$-sequential (distributional) topological complexity of $\text{conf}(n,w)$ when $n>w$.

\begin{lem}\label{lower bound for top complex conf}
If $n>w$, then
\[
r\Big(n-\Big\lceil\frac{n}{w}\Big\rceil\Big)\le\textbf{dTC}_{\textbf{r}}\big(\text{conf}(n,w)\big)\le \textbf{TC}_{\textbf{r}}\big(\text{conf}(n,w)\big).
\]
\end{lem}

Combining this lemma, Proposition \ref{upper bound for seq top of conf}, and Gonz\'{a}lez and Grant's Theorem 1.3 \cite[Theorem 1.3]{gonzalez2015sequential}, we have calculated the sequential topological complexity of $\text{conf}(n,w)$.

\begin{T1}
  \thmtext
\end{T1}

Additionally, we have proven the third case of 

\begin{T2}
  \thmtextone
\end{T2}

Since $\text{conf}(1,w)$ is contractible, it remains to handle the $n\le w$ case of Theorem \ref{seq dist top complexity of conf(n,w)}, which we now do.

\begin{proof}
For $n\le w$, consider the following two (products of) wheels in $\text{conf}(n,w)$
\[
A=W(n,\dots, 1)\indent\text{and}\indent B=W(n-1, \dots, 1)W(n).
\]
By Proposition \ref{decomposing concatenation products of wheels}, $A$ is a homologically decomposable $(n-1)$-torus and $B$ is a homologically decomposable $(n-2)$-torus.
Moreover, they are disjoint as the image of the summands of $A$ in $H_{1}\big(\text{conf}(n,w)\big)$ are of the form 
\[
W(n,j)+\cdots+W(n,n-1),
\]
for $1\le j\le n-1$, whereas none of the summands of $B$ in $H_{1}\big(\text{conf}(n,w)\big)$ have any term of the form $W(n,i)$.
Since the $W(m,l)$s form a basis for $H_{1}\big(\text{conf}(n,w)\big)$, it follows that $A$ and $B$ are homologically decomposable.
Therefore Corollary \ref{tori yield dtc lower bound} shows that for $1<n\le w$,
\[
\textbf{dTC}_{r}\big(\text{conf}(n,w)\big)\ge (r-1)(n-1)+n-2=r(n-1)-1.
\]
Since $\textbf{dTC}_{r}\big(\text{conf}(n,w)\big)\le \textbf{TC}_{r}\big(\text{conf}(n,w)\big)=r(n-1)-1$ for $1<n\le w$, we have the desired equality.
\end{proof}

Note that if $n\le w$, then $\text{conf}(n,w)\simeq F_{n}(\R^{2})$, so our proof of Theorem \ref{seq dist top complexity of conf(n,w)} also proves 

\begin{T3}
  \thmtexttwo
\end{T3}

\section{Open Questions}

We conclude this paper with several open questions about disk configuration spaces.

\begin{enumerate}

\item The quotient of $\text{conf}(n,w)$ by the obvious $S_{n}$ action yields \emph{$\text{uconf}(n,w)$}, the \emph{unordered configuration space of $n$ open unit-diameter disks in the infinite strip of width $w$}.
What is $(\textbf{d})\textbf{TC}_{\textbf{r}}\big(\text{uconf}(n,w)\big)$?
If $n$ is an odd number greater than $2$ and $w=2$, it is relatively easy to find two disjoint homologically decomposable tori of maximal dimension, proving the following proposition.

\begin{prop}\label{seq top comp of unordered conf(n,2) odd}
Let $2m+1$ be an odd number greater than $2$. 
Then
\[
(\textbf{d})\textbf{TC}_{\textbf{r}}\big(\text{uconf}(2m+1,2)\big)=rm.
\]
\end{prop}

For other $n$ and $w$, there is no obvious pair of disjoint homologically decomposable tori of maximal dimension.
Alpert and Manin give a basis for $H_{*}\big(\text{conf}(n,w);\F\big)$ when $\F=\Q$ or $\F=\F_{p}$ \cite[Theorems 9.3 and 9.7]{alpert2021configuration1}, though they do not describe its cohomology ring.
If $n\le w$, the unordered disk configuration space $\text{uconf}(n,w)$ is homotopy equivalent to $C_{n}(\R^{2})$, the \emph{unordered configuration space of $n$ points in $\R^{2}$}.
In this case, Bianchi and Recio-Mitter were able to find identical upper and lower bounds for the topological complexity of the unordered configuration space of points in the plane, proving that  $\textbf{TC}_{\textbf{2}}\big(C_{n}(\R^{2})\big)=2n-3$
\cite[Theorem 1.3 and 1.5]{bianchi2019topological}, \cite[Corollary 6.2]{bianchi2021homology}, and \cite[Theorem 1.5]{bianchi2020upper}.
While their methods rely on the fact that $C_{n}(\R^{2})$ is a $K(\pi,1)$ for the braid group, a fact that is not true for $\text{uconf}(n,w)$ for $n>w$, one would expect a similar result to hold given how $\text{uconf}(n,w)$ can be embedded in $C_{n}(\R^{2})$.

\item The ordered configuration space of $n$ open-unit diameter disks in the infinite strip of width $w$ is one of many hard shape configuration spaces.
Perhaps the most closely related is $F(n; p,q)$, the ordered configuration space of $n$ open unit-squares in the $p$ by $q$ rectangle.
What is $(\textbf{d})\textbf{TC}_{\textbf{r}}\big(F(n; p,q)\big)$?
Unlike $\text{conf}(n,w)$, the Betti numbers of $F(n; p,q)$, which have been studied in \cite{alpert2023asymptotic} and \cite{alpert2023homology}, are not always increasing $n$.
In fact, the $k^{\text{th}}$-homology of $F(n; p,q)$ is zero for $n$ sufficiently small \emph{or} large with respect to $pq$.
Moreover, it is believed that in the intermediate range these configuration spaces have homological dimension $\frac{pq}{4}$.
As such, does a sequential version of Farber's conjecture hold for $F(n; p,q)$, i.e., is $\frac{1}{r}\textbf{TC}_{\textbf{r}}\big(F(n; p,q)\big)$ independent of $n$ if $n$ is in this intermediate range?

\item
The (sequential) topological complexity of graph configuration spaces is relatively well studied, though there remain many open questions, e.g., \cite[Section 5]{knudsen2024farber}.
The most studied graph configuration spaces are those in which the underlying graph is a tree.
In this case, the (sequential) topological complexity of the ordered configuration space is known \cite[Theorem 1.2]{scheirer2018topological} and \cite[Theorem 5.8]{aguilar2022farley}, whereas in the unordered case sequential topological complexity is known if the number of points is sufficiently large \cite[Theorem 4.7]{hoekstra2023bounds}.
Many of these results use Abram's model for the discrete configuration space of a graph, see \cite[Section 2]{abrams2000configuration}, and the discrete vector field of Farley and Sabalka \cite{farley2008cohomology}.
What happens when we replace points with disks, i.e., insist that no pair of points can get within some finite distance of each other?
If the underlying graph is a tree, the Kim proved the configuration space is homotopy equivalent to a subcomplex of Abram's model for graph configuration space \cite[Theorem 3.2.10]{kim2022configuration}.
Can the restriction of Farley and Sabalka's discrete vector field be used to find the topological complexity of these configuration spaces?
\end{enumerate}

\bibliographystyle{amsalpha}
\bibliography{TopComplexConfnwBib}
\end{document}